\documentclass[12pt]{amsart}
\usepackage{graphicx}
\usepackage{xcolor}
\usepackage{amsmath,amssymb}
\numberwithin{equation}{section}

\newtheorem{theorem}{Theorem}[section]
\newtheorem{proposition}[theorem]{Proposition}
\newtheorem{lemma}[theorem]{Lemma}
\newtheorem{corollary}[theorem]{Corollary}
\theoremstyle{definition}

\theoremstyle{remark}
\newtheorem{remark}[theorem]{Remark}

\newcommand{\N}{\mathbb{N}}
\newcommand{\R}{\mathbb{R}}

\newcommand{\Aopt}{\mathcal A^{\mathrm{disc}}}

\title[Discrete Hardy and H\"older inequalities]{Discrete reversed Hardy inequalities and H\"older estimates in Lorentz sequence spaces}

\date{}
\author[S. Barza]{Sorina Barza}
\address{Department of Mathematics and Computer Science, Karlstad University, SE-65188 Karlstad, SWEDEN}
\email{sorina.barza@kau.se}

\author[A.N. Marcoci]{Anca-Nicoleta Marcoci}
\address{Department of Mathematics and Computer Science, Technical University of Civil Engineering Bucharest, RO-020396 Bucharest, ROMANIA}
\email{anca.marcoci@utcb.ro}

\author[L.G. Marcoci]{Liviu-Gabriel Marcoci}
\address{Department of Mathematics and Computer Science, Technical University of Civil Engineering Bucharest, RO-020396 Bucharest, ROMANIA}
\email{liviu.marcoci@utcb.ro}


\begin{document}

\begin{abstract}
We prove sharp reversed Hardy-type inequalities on the cone of decreasing sequences in weighted $\ell^p$-spaces with power weights, applying them to estimate standard norms in Lorentz sequence spaces. We prove that the discrete structure of these spaces yields optimal constants that are different from those in the continuous setting.  
\end{abstract}

\maketitle

\section{Introduction}

Let $1<p<\infty$ and $1\le s\le \infty$. The Lorentz sequence space $\ell^{p,s}$ consists of all complex sequences $x=(x_n)_{n\ge1}$ such that
\[
\|x\|_{p,s}
=
\begin{cases}
\left( \displaystyle \sum_{n=1}^{\infty} n^{\frac{s}{p}-1} (x_n^*)^s \right)^{1/s}, & 1\le s<\infty,\\[2ex]
\displaystyle \sup_{n\ge1} n^{1/p} x_n^*, & s=\infty,
\end{cases}
\]
is finite, where $x^*$ denotes the nonincreasing rearrangement of $(|x_n|)_n$.

Lorentz sequence spaces refine the classical scale of discrete Lebesgue spaces $\ell^p$ and play a central role in interpolation theory, harmonic analysis, and the geometry of Banach spaces. When $1\le s\le p$, the functional $\|\cdot\|_{p,s}$ is a norm, while for $p<s$ it is only a quasinorm. In this latter case, several equivalent norms are available. Among them, two are particularly relevant.

The first one is the maximal Lorentz norm
\[
\|x\|^{*}_{p,s}
=
\begin{cases}
\left( \displaystyle \sum_{n=1}^{\infty} n^{\frac{s}{p}-1} \left(\frac{1}{n}\sum_{k=1}^n x_k^*\right)^s \right)^{1/s}, & 1\le s<\infty,\\[2ex]
\displaystyle \sup_{n\ge1} n^{-1/p'} \sum_{k=1}^n x_k^*, & s=\infty,
\end{cases}
\]
where $1/p+1/p'=1$. The second one is the dual norm
\[
\|x\|_{p,s}'
=
\sup\left\{
\left|\sum_{n=1}^{\infty} x_n y_n\right|:
 y\in \ell^{p',s'},\; \|y\|_{p',s'}\le1
\right\},
\]
where $1/s+1/s'=1$, with the usual convention $s'=\infty$ if $s=1$.

We shall use the discrete Hardy operator, or Ces\`aro operator,
\[
(Cx)_n=\frac{1}{n}\sum_{k=1}^n |x_k|.
\]
For nonnegative nonincreasing sequences $x=(x_n)_n$, if
\begin{equation}\label{icsn}   
X_n=\sum_{k=1}^n x_k,
\end{equation}
then
\[
\|x\|_{p,s}^{*}
=
\left(\sum_{n=1}^{\infty} n^{-\frac{s}{p'}-1}X_n^s\right)^{1/s},
\qquad 1\le s<\infty.
\]

In the continuous setting, sharp relations between the standard, maximal, and dual norms in Lorentz spaces $L^{p,s}$ were established in \cite{BKS} and \cite{KolyadaSoria2010}. While these works completely solve the continuous case, a finding of the present paper is that the discrete atomic structure yields different optimal constants. Consequently, our sharp discrete H\" older-type inequalities and norm estimates are not adaptations of continuous bounds, but represent a distinct case.
The key tool is the level function introduced by Halperin and Lorentz in \cite{Ha} and \cite{Lo}. The level function was later studied systematically by Sinnamon, in particular in connection with duality; see \cite{Si1}, \cite{Si2} and the references given there.

In the discrete framework, the representation of the dual norm in $\ell^{p,s}$ via the level sequence was obtained in \cite{BarzaMarcociPersson2012} and \cite{BMMP}. This representation yields
\[
\|x\|_{p,s}' = \|x^{\circ}\|_{p,s},
\]
where $x^{\circ}$ denotes the level sequence associated with $x$ with respect to the weight
\begin{equation}\label{fiandalfa}
{\varphi}_n=n^{-\alpha},
\qquad
\alpha = 1 - \frac{s'}{p'}.
\end{equation}
This identity plays in $\ell^{p,s}$ the same role as the level function does in the continuous case. It is used here for two reasons. First, the dual norm is defined by a supremum over all test sequences, and the level-sequence theorem converts this external variational problem into the internal Lorentz norm of a canonically associated decreasing sequence. Secondly, the maximal norm is governed by Hardy averages of the original sequence. The comparison between the dual and maximal norms therefore becomes a comparison between two monotone sequences, the level sequence and the Hardy averages. This is the discrete analogue of the standard role of the Halperin--Lorentz level function in sharp Lorentz-space duality.

The present paper develops a discrete approach to norm estimates and H\"older-type inequalities in Lorentz sequence spaces. Building on the distinction established above, we show how the first atoms determine the new optimal constants. This contribution is given here by the one-atom sequence
\[
u^1=(1,0,0,\ldots).
\]
Indeed,
\[
\|u^1\|_{p,s}^{*\,s}
=\sum_{n=1}^{\infty} n^{-\frac{s}{p'}-1}
=\zeta\!\left(\frac{s}{p'}+1\right),
\]
and this value gives the sharp estimates between the dual and maximal Lorentz norms.

Our first result is a sharp reversed Hardy inequality with increasing power weights for nonincreasing  sequences. In the context of Lorentz sequence spaces we get
\[
\zeta\!\left(\frac{s}{p'}+1\right)^{1/s}\|x\|_{p,s}
\le \|x\|_{p,s}^{*},
\qquad 1<p\le s<\infty.
\]
The constant is optimal and is attained by the "one-point sequence".

For decreasing power weights, the optimal  constant in the discrete case is expressed as a supremum over the block sequences
\[u^K=(\underbrace{1,\ldots,1}_{K\text{ times}},0,0,\ldots).
\]
This formulation captures both the atomic behavior and the long-block asymptotic regime.

Combining this representation with the sharp reversed Hardy inequality gives
\[
\|x\|_{p,s}'
\le
\zeta\!\left(\frac{s}{p'}+1\right)^{-1/s}
\|x\|_{p,s}^{*},
\qquad 1<p\le s<\infty,
\]
with best possible constant.

Finally, we formulate H\"older-type inequalities involving maximal Lorentz norms. The optimal discrete H\"older constant is denoted as a variational constant, and we give computable lower bounds from the block sequences together with an explicit upper bound obtained from the one-factor norm comparisons.

The paper is organized as follows. Section~2 recalls basic definitions and the construction of the level sequence. Section~3 establishes reversed Hardy-type inequalities for decreasing sequences with power weights and derives the corresponding norm estimates. Section~4 proves the optimal equivalence between the dual and maximal norms. Section~5 contains the H\"older-type inequality in maximal Lorentz norms and a comparison between the discrete and continuous constants. We restrict ourselves to the case $1<p,s<\infty$.

\section{Definitions and notation}

\noindent For $p>1$ and $a\in\R$, by $\ell^p(n^a)$ we denote the weighted Lebesgue space of sequences defined by
\[
\|x\|_{\ell^p(n^a)}=\left(\sum_{n=1}^\infty |x_n|^p n^a\right)^{1/p}<\infty.
\]
We denote by $\ell^{p,\downarrow}(n^a)$ the cone of nonnegative,  nonincreasing sequences from $\ell^p(n^a)$.

Since all Lorentz norms considered in this paper are rearrangement invariant, it is enough to work with nonnegative nonincreasing sequences:
\[
   x_1\ge x_2\ge \cdots \ge 0.
\]

\noindent With the notation in \eqref{icsn}, for $1\le s<\infty$, we have
\[
   \|x\|_{p,s}^{*}
   := \left( \sum_{n=1}^\infty n^{\frac{s}{p}-1}
            \left(\frac{X_n}{n}\right)^s \right)^{1/s}
    = \left( \sum_{n=1}^\infty n^{-\frac{s}{p'}-1} X_n^s\right)^{1/s},
\]
and, for $s=\infty$,
\[
   \|x\|_{p,\infty}^{*}
   := \sup_{n\ge1} n^{-1/p'} X_n.
\]

For two nonnegative sequences $x=(x_n)_n$ and $z=(z_n)_n$ we write $x\prec z$ if
\[
\sum_{i=1}^{n}x_i\le\sum_{i=1}^{n}z_i,
\qquad n\ge 1.
\]

The \emph{level sequence} $x^\circ=(x_n^\circ)_n$ of $x$ with respect to $\varphi$ is defined in \cite[Section~3]{BarzaMarcociPersson2012}. We recall the following result.

\begin{theorem}[Theorem 3.3, \cite{BarzaMarcociPersson2012}]\label{level functions}
Let $\varphi=(\varphi_n)_{n}$ be a positive sequence and set
\[
\Phi_n=\sum_{i=1}^n \varphi_i.
\]
Let $x=(x_n)_n$ be a nonnegative sequence and suppose that
\begin{equation}\label{limita}
\lim_{n\rightarrow \infty} \frac{\sum_{i=1}^n x_i}{\Phi_n}=0.
\end{equation}
Then there exists a unique nonnegative sequence $x^\circ=(x^\circ_n)_n$ satisfying the following conditions:
\begin{enumerate}
\item[(a)] $\left(x_n^\circ/\varphi_n\right)_n$ is decreasing;
\item[(b)] $x \prec x^\circ$;
\item[(c)] the set $\{n:\, x^\circ_n\neq x_n\}$ is a union of disjoint intervals $I_k=\{M_k,\dots, N_k\}$ such that
\[
\sum_{i\in I_k}x_i=\sum_{i\in I_k}x_i^\circ
\]
and, for some constant $\lambda_k>0$,
\[
\frac{x_i^\circ}{\varphi_i}=\lambda_k,
\qquad i\in I_k.
\]
\end{enumerate}
\end{theorem}

The following representation of the dual norm was proved in \cite{BarzaMarcociPersson2012}.

\begin{theorem}[Theorem 4.2, \cite{BarzaMarcociPersson2012}]\label{t3}
Let $1<p<s\le \infty$ and let $x=(x_n)_n\in\ell^{p,s}$ be nonnegative and nonincreasing. Let
$\alpha$ and $\varphi$ be as defined in (\ref{fiandalfa}). 
Then
\[
\|x\|'_{p,s}=\|x^\circ\|_{p,s},
\]
where $x^\circ=(x^\circ_n)_n$ is the level sequence of $x$ with respect to $\varphi$.
\end{theorem}

\section{Reversed Hardy inequalities with power weights for decreasing sequences}

\noindent To establish inequalities relating the maximal Lorentz norm to the standard Lorentz norm, one needs reversed Hardy-type inequalities for nonincreasing sequences. In the discrete setting these are also known as Ces\`aro inequalities. For functions, related results are classical; see, for example, \cite{HLP} and \cite{KMP}. For sequences, the unweighted case appears in \cite{R}; discrete factorization methods for classical inequalities are systematically developed in \cite{B}. The weighted discrete case requires some special attention since increasing and decreasing weights lead to different constants.

\noindent In this section we write
\[
T_n(q)=\sum_{k=n+1}^{\infty}\frac{1}{k^q},
\qquad
S_n(q)=\sum_{k=1}^{n}\frac{1}{k^q},
\qquad q>1.
\]
For a fixed $K\ge1$, let $u^K=(u_n^K)_n$ be the sequence
\begin{equation}\label{unK}
 u_n^K=
 \begin{cases}
    1, & 1\le n\le K,\\
    0, & n>K.
    \end{cases}
\end{equation}

\subsection{Increasing power weights}

We begin with the case $0\le a<p-1$. We shall need the following lemma from \cite{R}.

\begin{lemma}[\cite{R}, Lemma 2]\label{Lem1R}
For $n\ge2$ and $q>1$,
\[
 \left(n^q-(n-1)^q-1\right)T_{n-1}(q)\ge S_{n-1}(q).
\]
\end{lemma}

\noindent As a direct consequence of Lemma~\ref{Lem1R}, we obtain the following weighted version.

\begin{lemma}\label{Lem1}
For $n\ge 2$, $p>1$, and $0\le a<p-1$,
\[
 \left(n^p-(n-1)^p-n^a\right)T_{n-1}(p-a)
 \ge n^a S_{n-1}(p-a).
\]
\end{lemma}

\begin{proof}
After division by $n^a$, the desired inequality becomes
\[
\left(n^{p-a}-(n-1)^{p-a}\frac{(n-1)^a}{n^a}-1\right)T_{n-1}(p-a)
\ge S_{n-1}(p-a).
\]
Since $a\ge0$, we have $(n-1)^a/n^a\le1$, and hence
\[
n^{p-a}-(n-1)^{p-a}\frac{(n-1)^a}{n^a}-1
\ge
n^{p-a}-(n-1)^{p-a}-1.
\]
As $p-a>1$, the desired estimate now follows from Lemma~\ref{Lem1R} applied with $q=p-a$.
\end{proof}

\noindent The following elementary lemma is the discrete convexity estimate used in the proof of the reversed Hardy inequality.

\begin{lemma}[\cite{R}, Lemma 1]\label{LemR}
Let $p>1$ and let $x_1\ge x_2\ge\cdots\ge x_n\ge0$. Then, for $n\ge2$,
\[
(x_1+x_2+\cdots+x_n)^p-(x_1^p+x_2^p+\cdots+x_n^p)
\ge
\sum_{k=2}^n (k^p-(k-1)^p-1)x_k^p.
\]
\end{lemma}

\noindent We shall use the following weighted analogue.

\begin{lemma}\label{LemR2}
Let $p>1$, $0\le a<p-1$, $n\ge2$, and let $x_1\ge x_2\ge\cdots\ge x_n\ge0$. Then
\begin{align*}
&(x_1+x_2+\cdots+x_n)^p-(x_1^p+2^a x_2^p+\cdots+n^a x_n^p) \\
&\hspace{2cm}\ge
\sum_{k=2}^n (k^p-(k-1)^p-k^a)x_k^p.
\end{align*}
\end{lemma}

\begin{proof}
By Lemma~\ref{LemR},
\[
(x_1+\cdots+x_n)^p-\sum_{k=1}^n x_k^p
\ge
\sum_{k=2}^n\bigl(k^p-(k-1)^p-1\bigr)x_k^p.
\]
Since
\[
\sum_{k=1}^n k^a x_k^p
=
\sum_{k=1}^n x_k^p+
\sum_{k=2}^n (k^a-1)x_k^p,
\]
we obtain
\begin{align*}
&(x_1+\cdots+x_n)^p-\sum_{k=1}^n k^a x_k^p \\
&\quad\ge
\sum_{k=2}^n\bigl(k^p-(k-1)^p-1\bigr)x_k^p
-
\sum_{k=2}^n(k^a-1)x_k^p \\
&\quad=
\sum_{k=2}^n\bigl(k^p-(k-1)^p-k^a\bigr)x_k^p,
\end{align*}
which completes the proof. 
\end{proof}

\noindent The next theorem is a weighted reversed Hardy inequality for increasing power weights. The case $a=0$ was proved in \cite{R}.

\begin{theorem}\label{thH1}
Let $p>1$ and $0\le a<p-1$. If $x\in \ell^{p,\downarrow}(n^a)$, then
\[
\zeta(p-a)^{-1/p}\|Cx\|_{\ell^{p}(n^a)}
\ge
\|x\|_{\ell^{p}(n^a)}.
\]
The constant $\zeta(p-a)^{-1/p}$ is optimal.
\end{theorem}

\begin{proof}
Let $x\in \ell^{p,\downarrow}(n^a)$ be nonnegative and nonincreasing. Put
\[
\varepsilon_1=0
\]
and, for $n\ge2$,
\[
\varepsilon_n=(x_1+\cdots+x_n)^p-(x_1^p+2^a x_2^p+\cdots+n^a x_n^p).
\]
Then,
\begin{align*}
\|Cx\|^p_{\ell^p(n^a)}
&=\sum_{n=1}^{\infty}\frac{(x_1+\cdots+x_n)^p}{n^{p-a}}\\
&=\sum_{n=1}^{\infty} n^{-(p-a)}\sum_{k=1}^n k^a x_k^p
  +\sum_{n=2}^{\infty}\frac{\varepsilon_n}{n^{p-a}}.
\end{align*}
Changing the order of summation in the first term gives
\begin{align*}
\|Cx\|^p_{\ell^p(n^a)}
&=\sum_{k=1}^{\infty} k^a x_k^p\sum_{n=k}^{\infty} n^{-(p-a)}
  +\sum_{n=2}^{\infty}\frac{\varepsilon_n}{n^{p-a}}\\
&=\zeta(p-a)\|x\|_{\ell^p(n^a)}^p
  -\sum_{n=2}^{\infty} n^a S_{n-1}(p-a)x_n^p
  +\sum_{n=2}^{\infty}\frac{\varepsilon_n}{n^{p-a}}.
\end{align*}
By Lemma~\ref{LemR2},
\[
\varepsilon_n\ge \sum_{k=2}^n (k^p-(k-1)^p-k^a)x_k^p.
\]
Therefore, we get
\begin{align*}
\sum_{n=2}^{\infty}\frac{\varepsilon_n}{n^{p-a}}
&\ge
\sum_{n=2}^{\infty}\frac{1}{n^{p-a}}
\sum_{k=2}^n (k^p-(k-1)^p-k^a)x_k^p\\
&=\sum_{n=2}^{\infty}(n^p-(n-1)^p-n^a)T_{n-1}(p-a)x_n^p.
\end{align*}
By Lemma~\ref{Lem1},
\[
(n^p-(n-1)^p-n^a)T_{n-1}(p-a)
\ge n^aS_{n-1}(p-a).
\]
Consequently,
\[
\|Cx\|^p_{\ell^p(n^a)}
\ge
\zeta(p-a)\|x\|_{\ell^p(n^a)}^p,
\]
which is the desired inequality.

\noindent We can see that the constant is sharp, by testing the inequality on $u^1=(1,0,0,\ldots)$. Indeed,
\[
\|u^1\|_{\ell^p(n^a)}=1
\]
and
\[
\|Cu^1\|^p_{\ell^p(n^a)}=
\sum_{n=1}^{\infty}n^{a-p}=\zeta(p-a).
\]
\end{proof}

\begin{corollary}\label{Cor36}
If $1<p\le s<\infty$ and $x\in\ell^{p,s}$, then
\begin{equation}\label{cor36}
\zeta\!\left(\frac{s}{p'}+1\right)^{1/s}\|x\|_{p,s}
\le
\|x\|^*_{p,s}.
\end{equation}
The constant is optimal.
\end{corollary}

\begin{proof}
The inequality follows from Theorem~\ref{thH1} applied with    
\[
a=\frac{s}{p}-1\ge0.
\]
Indeed,
\[
s-a=s-\left(\frac{s}{p}-1\right)=\frac{s}{p'}+1.
\]
The sharpness of the constant follows by testing with $u^1$, as in Theorem~\ref{thH1}.
\end{proof}

\subsection{Decreasing power weights}

The case $-1<a<0$ is different. In the continuous setting one obtains a simple limiting constant. In the discrete setting, however, the first atoms may give a larger ratio. For this reason the exact discrete constant is   expressed as a supremum over the block sequences $u^K$.
Reversed Hardy inequalities for nonincreasing sequences with power weights appear also in Bennett and Grosse-Erdmann \cite{BGE2015}; in particular, Theorem~3 there is closely related to Theorem~\ref{thH2} below. Our point here is the identification of the exact discrete constant as a supremum over blocks, and the fact that it generally differs from the continuous constant; see Remark~\ref{remark-h2-constant}.

For $r>1$ and $-1<a<0$, we define
\begin{equation}\label{def-hra}
\mathcal H_{r,a}^r
:=
\sup_{K\ge1}
\frac{\displaystyle\sum_{n=1}^K n^a}
{\displaystyle\sum_{n=1}^K n^a+K^r\sum_{n=K+1}^{\infty}n^{a-r}}.
\end{equation}
Then $0<\mathcal H_{r,a}<1$. Moreover, we can easily see that
\begin{equation}\label{h-lower-bounds}
\mathcal H_{r,a}^r
\ge
\max\left\{
\frac{1}{\zeta(r-a)},\frac{r-a-1}{r}
\right\},
\end{equation}
which follows, respectively, from $K=1$ and from the limit $K\to\infty$.

\begin{lemma}\label{layer-average-lemma}
Let $r>1$ and let $a_1,\ldots,a_n\ge0$. Then
\[
r\int_0^{\infty}t^{r-1}
\left(\sum_{k=1}^n \mathbf 1_{\{a_k>t\}}\right)^r\,dt
\le
\left(\sum_{k=1}^n a_k\right)^r.
\]
Consequently,
\[
r\int_0^{\infty}t^{r-1}
\left(\frac1n\sum_{k=1}^n \mathbf 1_{\{a_k>t\}}\right)^r\,dt
\le
\left(\frac1n\sum_{k=1}^n a_k\right)^r.
\]
\end{lemma}

\begin{proof}
After rearranging the numbers, assume $a_1\ge\cdots\ge a_n\ge0$ and put $a_{n+1}=0$. Then
\begin{align*}
r\int_0^{\infty}t^{r-1}
\left(\sum_{k=1}^n \mathbf 1_{\{a_k>t\}}\right)^r\,dt
&=\sum_{j=1}^n j^r(a_j^r-a_{j+1}^r)\\
&=\sum_{j=1}^n (j^r-(j-1)^r)a_j^r.
\end{align*}
By Lemma~\ref{LemR}, applied with $p=r$, the last expression is bounded above by $(a_1+\cdots+a_n)^r$.
\end{proof}

\begin{theorem}\label{thH2}
Let $r>1$ and $-1<a<0$. If $x\in\ell^{r,\downarrow}(n^a)$, then
\[
\mathcal H_{r,a}\|Cx\|_{\ell^r(n^a)}
\ge
\|x\|_{\ell^r(n^a)}.
\]
The constant $\mathcal H_{r,a}$ is optimal.
\end{theorem}

\begin{proof}
We follow the layer-cake argument used in \cite{BPS}. Let $x\in\ell^{r,\downarrow}(n^a)$ be nonnegative. For $t>0$, put
\[
D_t=\{n\in\N:\ x_n>t\}.
\]
Since $x$ is nonincreasing, either $D_t=\varnothing$ or $D_t=\{1,2,\ldots,N_t\}$ for some integer $N_t\ge1$.

By Fubini's theorem,
\begin{align*}
\|x\|_{\ell^r(n^a)}^r
&=\sum_{n=1}^{\infty}x_n^r n^a
=r\sum_{n=1}^{\infty}n^a\int_0^{x_n}t^{r-1}\,dt\\
&=r\int_0^{\infty}t^{r-1}\sum_{n\in D_t}n^a\,dt
=r\int_0^{\infty}t^{r-1}\|u^{N_t}\|_{\ell^r(n^a)}^r\,dt.
\end{align*}
By the definition of $\mathcal H_{r,a}$,
\[
\|u^K\|_{\ell^r(n^a)}
\le
\mathcal H_{r,a}\|Cu^K\|_{\ell^r(n^a)},
\qquad K\ge1.
\]
Hence
\begin{align*}
\|x\|_{\ell^r(n^a)}^r
&\le
\mathcal H_{r,a}^r r\int_0^{\infty}t^{r-1}\|Cu^{N_t}\|_{\ell^r(n^a)}^r\,dt\\
&=\mathcal H_{r,a}^r r\int_0^{\infty}t^{r-1}
\sum_{n=1}^{\infty}\left(\frac1n\sum_{k=1}^n u_k^{N_t}\right)^r n^a\,dt\\
&\le\mathcal H_{r,a}^r
\sum_{n=1}^{\infty}n^a
\left(\frac1n\sum_{k=1}^n x_k\right)^r.
\end{align*}
In the last step we used Lemma~\ref{layer-average-lemma} for each fixed $n$. Thus
\[
\|x\|_{\ell^r(n^a)}^r
\le
\mathcal H_{r,a}^r\|Cx\|_{\ell^r(n^a)}^r.
\]

The constant is optimal since
\[
\sup_{K\ge1}\frac{\|u^K\|_{\ell^r(n^a)}}{\|Cu^K\|_{\ell^r(n^a)}}=\mathcal H_{r,a},
\]
 by \eqref{def-hra}.
\end{proof}

\begin{remark}\label{remark-h2-constant}
The limiting value
\[
\lim_{K\to\infty}
\frac{\displaystyle\sum_{n=1}^K n^a}
{\displaystyle\sum_{n=1}^K n^a+K^r\sum_{n=K+1}^{\infty}n^{a-r}}
=\frac{r-a-1}{r}
\]
is only an asymptotic  bound for $\mathcal H_{r,a}^r$. It is not, in general, the supremum in the discrete problem. The one-atom value gives
\[
\frac{\|u^1\|_{\ell^r(n^a)}^r}{\|Cu^1\|_{\ell^r(n^a)}^r}
=\frac1{\zeta(r-a)}.
\]
Thus the discrete constant may be larger than the  constant in the continuous case. This is the phenomenon that must be taken into account in the Lorentz sequence estimates below.
\end{remark}

\begin{corollary}\label{cormain}
Let $1<s<p<\infty$ and set
\begin{equation}\label{def-hps}
\mathcal C_{p,s}:=\mathcal H_{s,\,s/p-1}.
\end{equation}
Then, for every $x\in\ell^{p,s}$,
\begin{equation}\label{cormain-discrete}
\mathcal C_{p,s}^{-1}\|x\|_{p,s}
\le
\|x\|^*_{p,s}.
\end{equation}
The constant is optimal.
\end{corollary}

\begin{proof}
The inequality follows from Theorem~\ref{thH2} with $r=s$ and weight exponent
\[
a=\frac{s}{p}-1\in(-1,0).
\]
The optimality of the constant follows from the optimality of $\mathcal H_{s,s/p-1}$ in Theorem~\ref{thH2}.
\end{proof}

\begin{remark}[The converse estimate]\label{rem-upper-hardy}
Inequalities \eqref{cor36} and \eqref{cormain-discrete} compare the standard and the maximal Lorentz norms for non-negative sequences. An estimate of the form
\[
\|x\|^*_{p,s}\le A_{p,s}\|x\|_{p,s}
\]
holds, for all $1<p,s<\infty$, with some finite constant $A_{p,s}$. Indeed, on the cone of non-negative nonincreasing sequences this is a weighted discrete Hardy inequality with the weight $b_n=n^{s/p-1}$, and this weight satisfies the discrete condition of Ari\~no--Muckenhoupt type which characterizes such inequalities; see \cite[Theorem~1]{BGE2006}. For $p=s$ the estimate reduces to the classical discrete Hardy inequality and the best constant is $A_{p,p}=p'$; see \cite{HLP}. For $p\neq s$ the test sequences $x_n=n^{-1/p-\varepsilon}$, $\varepsilon\to 0^{+}$, show that no constant smaller than $p'$ is admissible, but the value of the best constant seems to be unknown; compare \cite[Remark~2]{BGE2006} and \cite[Corollary, p.~270]{HK}. Since this estimate is not used in the proofs of the results of this paper, we do not pursue this question here.
\end{remark}

\section{Equivalence between the dual and maximal norm}

We start by computing the Lorentz norms of the block sequences $u^K$.

\begin{lemma}\label{lema41}
Let $u^K=(u_n^K)_n$ be defined by \eqref{unK}. If $1<p,s<\infty$, then
\[
{\|u^K\|^*}_{p,s}^s
=
\sum_{n=1}^K n^{\frac{s}{p}-1}
+K^s\sum_{n=K+1}^{\infty}n^{-\frac{s}{p'}-1}.
\]
Moreover, if $1<p\le s<\infty$, then
\[
{\|u^K\|'}_{p,s}^s
=K^s\left(\sum_{n=1}^K n^{-\alpha}\right)^{-s/s'},
\qquad
\alpha=1-\frac{s'}{p'}.
\]
If $1<p<s<\infty$, we have  
\[
{\|(u^K)^\circ\|}_{p,s}^s
=K^s\left(\sum_{n=1}^K n^{-\alpha}\right)^{-s/s'},
\]
where $(u^K)^\circ$ is the level sequence of $u^K$ with respect to
${\varphi}_n=n^{-\alpha}$.
\end{lemma}

\begin{proof}
The formula for the maximal norm follows from
\[
\sum_{k=1}^n u_k^K=\min\{n,K\}.
\]
Thus
\[
{\|u^K\|^*}_{p,s}^s
=
\sum_{n=1}^K n^{-\frac{s}{p'}-1}n^s
+\sum_{n=K+1}^{\infty}n^{-\frac{s}{p'}-1}K^s,
\]
which is the first formula.

For the second formula, the level sequence of $u^K$ with respect to $\varphi(n)=n^{-\alpha}$ is supported on $\{1,\ldots,K\}$ and is given by
\[
(u^K)_n^\circ
=
\frac{K n^{-\alpha}}{\sum_{j=1}^K j^{-\alpha}},
\qquad 1\le n\le K.
\]
Therefore
\begin{align*}
{\|(u^K)^\circ\|}_{p,s}^s
&=
\sum_{n=1}^K n^{\frac{s}{p}-1}
\left(\frac{K n^{-\alpha}}{\sum_{j=1}^K j^{-\alpha}}\right)^s\\
&=K^s\left(\sum_{j=1}^K j^{-\alpha}\right)^{-s}
\sum_{n=1}^K n^{\frac{s}{p}-1-\alpha s}.
\end{align*}
Since
\[
\frac{s}{p}-1-\alpha s=-\alpha,
\]
we obtain
\[
{\|(u^K)^\circ\|}_{p,s}^s
=K^s\left(\sum_{n=1}^K n^{-\alpha}\right)^{1-s}
=K^s\left(\sum_{n=1}^K n^{-\alpha}\right)^{-s/s'}.
\]
If $1<p<s<\infty$, the identity with the dual norm follows from
Theorem~\ref{t3}. If $p=s$, then $\ell^{p,p}=\ell^p$ isometrically, and the same
formula follows from the usual $\ell^p$-duality.
\end{proof}

\noindent In the next theorem, we prove inequalities  between the dual and the maximal norm. The formulation with the level sequence is not merely notational. The dual norm is a supremum over all admissible test sequences, whereas the maximal norm is expressed through Ces\`aro averages. The norm of the level sequence is equal with the dual norm and therefore allows the supremum to be replaced by a concrete Lorentz norm. Thus, in the range $1<p\le s<\infty$, the estimate is written explicitly in terms of the level-sequence.

\begin{theorem}\label{thBps}
Let $1<p\le s<\infty$. Then, for every $x\in\ell^{p,s}$,
\begin{equation}\label{dual-max-discrete}
\|x\|'_{p,s}
\le
\zeta\!\left(\frac{s}{p'}+1\right)^{-1/s}
\|x\|^*_{p,s}.
\end{equation}
The constant is optimal.

If $1<p<s<\infty$ and $x^\circ$ denotes the level sequence of $x$ with respect to
\[
{\varphi}_{n}=n^{-\alpha},\qquad \alpha=1-\frac{s'}{p'},
\]
then the same estimate can be written in the level-sequence form
\begin{equation}\label{level-dual-max-estimate}
\|x^\circ\|_{p,s}
\le
\zeta\!\left(\frac{s}{p'}+1\right)^{-1/s}
\|x\|^*_{p,s}.
\end{equation}
\end{theorem}

\begin{proof}
By rearrangement invariance, we may assume that $x$ is nonnegative and nonincreasing. If $1<p<s<\infty$, Theorem~\ref{t3} gives
\[
\|x\|'_{p,s}=\|x^\circ\|_{p,s},
\]
where $x^\circ$ is the level sequence associated with $x$. Thus the level-sequence representation reduces the estimate for the dual norm to an estimate for the Lorentz norm of $x^\circ$.

For every $y\in\ell^{p',s'}$, the rearrangement inequality and H\"older's inequality with the conjugate weights give
\begin{align*}
|\sum_{n=1}^{\infty}x_n y_n|\le\sum_{n=1}^{\infty}x_n^* y_n^*
&\le
\left(\sum_{n=1}^{\infty}n^{\frac{s}{p}-1}(x_n^*)^s\right)^{1/s}
\left(\sum_{n=1}^{\infty}n^{\frac{s'}{p'}-1}(y_n^*)^{s'}\right)^{1/s'} \\
&=\|x\|_{p,s}\|y\|_{p',s'}.
\end{align*}
Taking the supremum over all $y$ satisfying $\|y\|_{p',s'}\le1$ yields
\[
\|x\|'_{p,s}\le \|x\|_{p,s}.
\]
If $p=s$, the same conclusion follows directly from the usual weighted H\"older inequality.

By Corollary~\ref{Cor36},
\[
\|x\|_{p,s}
\le
\zeta\!\left(\frac{s}{p'}+1\right)^{-1/s}\|x\|^*_{p,s}.
\]
Combining the last two inequalities proves \eqref{dual-max-discrete}. In the case $1<p<s<\infty$, Theorem~\ref{t3} also gives the equivalent level-sequence formulation \eqref{level-dual-max-estimate}.

To prove sharpness, take $x=u^1=(1,0,0,\ldots)$. Then
\[
\|u^1\|'_{p,s}=1,
\]
whereas
\[
\|u^1\|_{p,s}^{*\,s}
=
\sum_{n=1}^{\infty}n^{-\frac{s}{p'}-1}
=
\zeta\!\left(\frac{s}{p'}+1\right).
\]
Thus, equality holds in \eqref{dual-max-discrete} for $u^1$, and the constant is optimal.
\end{proof}

\begin{remark}[Comparison with the non-atomic constant]\label{continuous-constant-remark}
For $1<p<s<\infty$, the continuous analogue (see \cite{KolyadaSoria2010}) has the constant
\[
B_{p,s}^{\mathrm{cont}}
=
(p')^{-1/s}
\left(\frac{s'}{p'}\right)^{1/s'}
\left(\frac{s}{p}\right)^{1/s}.
\]
This number is recovered in the discrete setting as an asymptotic constant. Indeed, by Lemma~\ref{lema41},
\[
{\|u^K\|'}_{p,s}^{s}
=K^s\left(\sum_{n=1}^K n^{-\alpha}\right)^{-s/s'},
\]
\[
{\|u^K\|^*}_{p,s}^{s}
=\sum_{n=1}^K n^{\frac{s}{p}-1}
+K^s\sum_{n=K+1}^{\infty}n^{-\frac{s}{p'}-1},
\]
and, asymptotically, as $K\to\infty$,
\[
\sum_{n=1}^K n^{-\alpha}\sim\frac{p'}{s'}\,K^{s'/p'},
\qquad
\sum_{n=1}^K n^{\frac{s}{p}-1}\sim\frac{p}{s}\,K^{s/p},
\]
and
\[
K^s\sum_{n=K+1}^{\infty}n^{-\frac{s}{p'}-1}\sim\frac{p'}{s}\,K^{s/p}.
\]
Since $\frac{s'}{p'}\cdot\frac{s}{s'}=\frac{s}{p'}$ and $s-\frac{s}{p'}=\frac{s}{p}$, we obtain
\[
\lim_{K\to\infty}
\frac{{\|u^K\|'}_{p,s}^{s}}{{\|u^K\|^*}_{p,s}^{s}}
=
\left(\frac{s'}{p'}\right)^{s/s'}\frac{s}{p+p'}
=
\left(\frac{s'}{p'}\right)^{s/s'}\frac{s}{pp'}
=\bigl(B_{p,s}^{\mathrm{cont}}\bigr)^{s}.
\]
 Thus
\[
\lim_{K\to\infty}
\frac{\|u^K\|'_{p,s}}{\|u^K\|^*_{p,s}}
=
B_{p,s}^{\mathrm{cont}}.
\]
However, this is not the best constant in the discrete case. The one-atom sequence gives
\[
\frac{\|u^1\|'_{p,s}}{\|u^1\|^*_{p,s}}
=
\zeta\!\left(\frac{s}{p'}+1\right)^{-1/s},
\]
and Theorem~\ref{thBps} shows that this  value is the optimal  constant in the discrete case.
\end{remark}

For completeness, we also mention the corresponding estimate in the range $1<s\le p<\infty$.

\begin{lemma}\label{norm-range-dual}
Let $1<s\le p<\infty$. Then
\[
\|x\|'_{p,s}=\|x\|_{p,s},
\qquad x\in\ell^{p,s}.
\]
\end{lemma}

\begin{proof}
The inequality $\|x\|'_{p,s}\le \|x\|_{p,s}$ follows from the weighted H\"older inequality used in the proof of Theorem~\ref{thBps}. For the reverse inequality, assume first that $x=x^*$ is nonnegative and finitely supported. Set
\[
y_n=\|x\|_{p,s}^{1-s}\, n^{\frac{s}{p}-1}x_n^{s-1}.
\]
Then $\|y\|_{p',s'}=1$. Since $s\le p$, the sequence $n^{s/p-1}$ is nonincreasing, and therefore $y$ is nonincreasing. A direct computation gives equality in H\"older's inequality. The general case follows by truncation and rearrangement invariance.
\end{proof}

\begin{proposition}\label{prop-dual-max-sp}
Let $1<s<p<\infty$ and let $\mathcal C_{p,s}$ be defined by \eqref{def-hps}. Then
\[
\|x\|'_{p,s}\le \mathcal C_{p,s}\|x\|^*_{p,s}.
\]
The constant is optimal.
\end{proposition}

\begin{proof}
By Lemma~\ref{norm-range-dual}, $\|x\|'_{p,s}=\|x\|_{p,s}$. Corollary~\ref{cormain} gives
\[
\|x\|_{p,s}\le \mathcal C_{p,s}\|x\|_{p,s}^{*}.
\]
The optimality follows from the optimality of the constant in Corollary~\ref{cormain}.
\end{proof}

\section{H\"older-type inequalities}

We now formulate  the H\"older-type inequality in terms of maximal Lorentz norms. 

Recall that for $1<s<p<\infty$ the constant $\mathcal C_{p,s}=\mathcal H_{s,\,s/p-1}$
is defined in \eqref{def-hps}. For $1<p\le s<\infty$ we define
$\mathcal C_{p,s}:=\zeta\!\left(\frac{s}{p'}+1\right)^{-1/s}$, so that, by
Corollaries~\ref{Cor36} and~\ref{cormain},
\begin{equation}\label{standard-max-embedding}
\|x\|_{p,s}\le \mathcal C_{p,s}\,\|x\|^{*}_{p,s},
\qquad 1<p,s<\infty.
\end{equation}

\begin{theorem}\label{Holder-explicit}
Let $1<p,s<\infty$. Then, for all $x\in\ell^{p,s}$ and $y\in\ell^{p',s'}$, we have
\begin{equation}\label{Holder-dps}
\left|\sum_{n=1}^{\infty}x_ny_n\right|
\le
\mathcal D_{p,s}
\|x\|^*_{p,s}\|y\|^*_{p',s'},
\end{equation}
where $\mathcal D_{p,s}=\max\left\{\mathcal C_{p,s}\zeta{\left(\frac{s'}{p}+1\right)}^{-1/s'},\mathcal C_{p',s'}\zeta{\left(\frac{s}{p'}+1\right)}^{-1/s}\right\}$.
\end{theorem}

\begin{proof}
Assume first that $1<p\le s<\infty$; in this case $s'\le p'$.
Applying the inequality of Theorem~\ref{thBps} and the definition of the dual norm, we get
\[
\left|\sum_{n=1}^{\infty}x_ny_n\right|
\le
\zeta{\left(\frac{s}{p'}+1\right)}^{-1/s}\|x\|^*_{p,s}\|y\|_{p',s'}.
\]
Using \eqref{standard-max-embedding} for the pair $(p',s')$ we obtain
\[
\left|\sum_{n=1}^{\infty}x_ny_n\right|
\le
\mathcal C_{p',s'}\zeta{\left(\frac{s}{p'}+1\right)}^{-1/s}\|x\|^*_{p,s}\|y\|^*_{p',s'}.
\]
If $1<s\le p<\infty$, then $p'\le s'$, and the same argument, with
Theorem~\ref{thBps} applied to the pair $(p',s')$ and
\eqref{standard-max-embedding} applied to the pair $(p,s)$, gives
\[
\left|\sum_{n=1}^{\infty}x_ny_n\right|
\le
\mathcal C_{p,s}\zeta{\left(\frac{s'}{p}+1\right)}^{-1/s'}\|x\|^*_{p,s}\|y\|^*_{p',s'}.
\]
Hence
\[
\left|\sum_{n=1}^{\infty}x_ny_n\right|
\]
\[
\le
\max\left\{\mathcal C_{p,s}\zeta{\left(\frac{s'}{p}+1\right)}^{-1/s'},
\mathcal C_{p',s'}\zeta{\left(\frac{s}{p'}+1\right)}^{-1/s}\right\}
\|x\|^*_{p,s}\|y\|^*_{p',s'},
\]
i.e.\ \eqref{Holder-dps}, which completes the proof. 
\end{proof}

\noindent The optimal constant in the H\"older's inequality in the discrete case associated with the maximal Lorentz norms i.e. 
\begin{equation}\label{Aopt-def}
\Aopt_{p,s}
:=
\sup_{x,y\ne0}
\frac{\left|\sum_{n=1}^{\infty}x_ny_n\right|}
{\|x\|^*_{p,s}\|y\|^*_{p',s'}}
\end{equation}
is not known.

\noindent With this notation, Theorem~\ref{Holder-explicit} gives the upper bound
\begin{equation}\label{A-upper}
\Aopt_{p,s}
\le
\mathcal D_{p,s}.
\end{equation}
On the other hand, the block sequences provide explicit lower bounds. For every $K\ge1$,
\begin{equation}\label{block-lower}
\Aopt_{p,s}
\ge
\frac{K}
{\|u^K\|^*_{p,s}\|u^K\|^*_{p',s'}}.
\end{equation}
Using the explicit formula for the maximal norm of $u^K$ in Lemma~\ref{lema41}, this becomes
\begin{align}\label{block-lower-expanded}
\Aopt_{p,s}
&\ge
K\left(\sum_{n=1}^{K}n^{\frac{s}{p}-1}
+K^s\sum_{n=K+1}^{\infty}n^{-\frac{s}{p'}-1}\right)^{-1/s}
\notag\\
&\quad\times
\left(\sum_{n=1}^{K}n^{\frac{s'}{p'}-1}
+K^{s'}\sum_{n=K+1}^{\infty}n^{-\frac{s'}{p}-1}\right)^{-1/s'}.
\end{align}
In particular, taking $K=1$ gives the one-atom lower bound
\begin{equation}\label{atom-holder-lower}
\Aopt_{p,s}
\ge
\zeta\!\left(\frac{s}{p'}+1\right)^{-1/s}
\zeta\!\left(\frac{s'}{p}+1\right)^{-1/s'}.
\end{equation}
Letting $K\to\infty$ in \eqref{block-lower-expanded} and using the power-sum asymptotics from Remark~\ref{continuous-constant-remark} (together with $p+p'=pp'$) gives the continuous asymptotic lower bound
\begin{equation}\label{cont-holder-lower}
\Aopt_{p,s}
\ge
\frac{s^{1/s}(s')^{1/s'}}{pp'}.
\end{equation}
Thus
\begin{equation}\label{holder-bounds-summary}
\max\left\{
\zeta\!\left(\frac{s}{p'}+1\right)^{-1/s}
\zeta\!\left(\frac{s'}{p}+1\right)^{-1/s'},
\frac{s^{1/s}(s')^{1/s'}}{pp'}
\right\}
\le
\Aopt_{p,s}
\le
\mathcal D_{p,s}.
\end{equation}

\begin{remark}\label{Holder-constant-remark}
The value $s^{1/s}(s')^{1/s'}/(pp')$ in \eqref{cont-holder-lower} is the sharp
constant in the continuous H\"older inequality of Kolyada and Soria
\cite{KolyadaSoria2010}. In the discrete setting it arises only as a long-block
limit and, by \eqref{atom-holder-lower}, it does not give the value of
$\Aopt_{p,s}$ in general. 

For $p=s$, for instance, the one-atom bound reads
$$\Aopt_{p,p}\ge\zeta(p)^{-1/p}\zeta(p')^{-1/p'},$$ which is strictly larger than
$p^{1/p}(p')^{1/p'}/(pp')$, since $\zeta(p)<p'$ and $\zeta(p')<p$. As in
Remark~\ref{remark-h2-constant}, the discrepancy comes from the contribution of
the first atoms. In fact, for $p=s$ the upper bound in
\eqref{holder-bounds-summary} coincides with the one-atom bound, so that
$$\Aopt_{p,p}=\zeta(p)^{-1/p}\,\zeta(p')^{-1/p'}.$$
\end{remark}

\end{document}